\documentclass[10pt,psamsfonts]{amsart}

\usepackage{amsthm}
\usepackage{amssymb}
\usepackage{amsmath}
\usepackage{graphicx}
\usepackage{amscd}
\usepackage{amsfonts}
\usepackage{amsbsy}
\usepackage[T1]{fontenc}
\usepackage[english]{babel}

\usepackage{epsfig}
\usepackage{amssymb}

\newtheorem{thm}{Theorem}
\newtheorem{cor}[thm]{Corollary}

\newtheorem{definition}{Definition}[section]
\newtheorem{lemma}[thm]{Lemma}

\title[Entropy, POTP and positively expansive measures]{Entropy, pseudo-orbit tracing property
and positively expansive measures}

\author{C.A. Morales}

\address{Instituto de Matem\'atica, Universidade Federal do Rio de Janeiro, P. O.
Box 68530, 21945-970 Rio de Janeiro, Brazil.}
\email{morales@impa.br.}



\begin{document}

\begin{abstract}
We study homeomorphisms of compact metric spaces
whose restriction to the nonwandering set has the
pseudo-orbit tracing property.
We prove that if there are positively expansive measures, then the topological entropy is positive.
Some short applications of this result are included.
\end{abstract}

\maketitle

\section{Introduction}

\noindent
It was proved in \cite{am} that
every continuous map with positive topological entropy of a compact metric space
exhibits a positively expansive measure.
It is then natural to ask if, conversely,
every continuous map exhibiting positively expansive measures of a compact metric space
has positively entropy.
However, the answer is negative since the Denjoy map \cite{d} is a circle homeomorphism with
both positively expansive measures and zero topological entropy \cite{ms}.
This makes us to re-ask:
Which homeomorphisms exhibiting positively expansive measures of a compact metric space have positive topological entropy?

In this paper we will address this last question. Indeed, we prove that every
homeomorphism exhibiting
positively expansive measures has positive topological entropy as soon as
its restriction to the nonwandering set has the
{\em pseudo-orbit tracing property} (POTP).
Some short applications are included. The proof will use some recent results about the structure of the POTP \cite{M}.
Our result can be added to the number of interesting properties
of the expansive measures (more properties are explained in \cite{ms}).
Let us state our result in a precise way.

Consider a homeomorphism $f: X\to X$ of a metric space $X$, and a bi-infinite sequence $(x_i)_{i\in\mathbb{Z}}\in X^\mathbb{Z}$.
Given $\delta\geq0$ we say that $(x_i)_{i\in\mathbb{Z}}$ is
a {\em $\delta$-pseudo-orbit} whenever $d(f(x_i),x_{i+1})\leq\delta$ for every $i\in\mathbb{Z}$.
We say that $(x_i)_{i\in\mathbb{Z}}$ can be {\em $\delta$-shadowed} if
there is $x\in X$ such that $d(f^i(x),x_i)\leq\delta$ for every $i\in\mathbb{Z}$.
We say that $f$ has the {\em pseudo-orbit tracing property} ({\em POTP} for short)
if for every $\epsilon>0$ there is $\delta>0$ such that
every $\delta$-pseudo-orbit of $f$ can be $\epsilon$-shadowed.
Analogously we define the POTP$_+$ for continuous maps instead of homeomosphisms
by replacing $\mathbb{Z}$ by $\mathbb{N}$ in the definition of POTP. It is well-known that
POTP and POTP$_+$ are equivalent for homeomorphisms \cite{p}.

On the other hand, the {\em topological entropy} of $f$
is defined by
$$
h(f)=\displaystyle\lim_{\epsilon\to0}\displaystyle\limsup_{n\to\infty}\frac{1}{n}\log M_f(n,\epsilon),
$$
where $M_f(n,\epsilon)$ is the minimal number of $\epsilon$-balls in the metric
$$
(x,y)\to \displaystyle\max_{0\leq i\leq n}d(f^i(x),f^i(y))
$$
needed to cover $X$ (see \cite{w0} for details).

We say that a (nonnecessarily invariant) Borel probability measure $\mu$ of $X$ is {\em positively expansive}
if there is $e>0$ (called {\em expansivity constant}) such that
$\mu(\Phi_e(x))=0$ for every $x\in X$, where
$$
\Phi_\delta(x)=\{y\in X:d(f^i(x),f^i(y))\leq r, \mbox{ for every } i\in\mathbb{N}\},
\quad\quad\forall x\in X,\delta\geq0.
$$
Recall that the {\em nonwandering set} $\Omega(f)$ of
a map $f: X\to X$ is
the set of points $x\in X$ such that for every neighborhood $U$ of $x$ there is $n\in\mathbb{N}^+$ satisfying
$f^n(U)\cap U\neq\emptyset$. Clearly if $f$ is a homeomorphism, then
$\Omega(f)$ is compact and invariant, i.e., $f(\Omega(f))=\Omega(f)$.

Our main result below gives a different condition
under which the POTP implies positive entropy (for more conditions see Section 5 in \cite{M}):

\begin{thm}
\label{thAA}
Let $f$ be a homeomorphism of a compact metric space
such that $f|_{\Omega(f)}$ has the POTP.
If $f$ has a positively expansive measure, then $f$ has positive topological entropy.
\end{thm}

Since the POTP for homeomorphisms implies
that of the corresponding restriction to the nonwandering set \cite{a}, we obtain
the following corollary from Theorem \ref{thAA}.

\begin{cor}
\label{thA}
Every homeomorphism with the POTP exhibiting positively expansive measures
of a compact metric space has positive topological entropy.
\end{cor}

We know from Lemma 3.1.2 in \cite{p} that
every orientation-preserving circle homeomorphism with the POTP exhibits periodic points
(actually it has at least two, see Theorem 3.1.1 in \cite{p}).
The following is a slight improvement of this result.

\begin{cor}
\label{c1}
Every circle homeomorphism $f$ for which
$f|_{\Omega(f)}$ has the POTP exhibits periodic point.
\end{cor}

\begin{proof}
Suppose by contradiction that $f$ has no periodic points.
Then, $f$ is topologically conjugate to either a rigid rotation by
an irrational angle or a Denjoy map \cite{d}.
In the first case we have that $\Omega(f)$ is the entire circle
and $f$ does not have the POTP.
Thus, we can assume that $f$ is the Denjoy map and so has
positively expansive measures \cite{ms}.
Then, the entropy would be positive  by Theorem \ref{thAA} which is absurd
for circle homeomorphisms.
This proves the result.
\end{proof}

Every surface homeomorphism with the POTP exhibits periodic points \cite{kp}.
This motivates the question if Corollary \ref{c1} is true for surface homeomorphisms instead of circle
homeomorphisms.

For the second application, let us remember that a homeomorphism $f$ of a metric space $X$ is {\em topologically stable} if
for every $\epsilon>0$ there is $\delta>0$ such that for every homeomorphism $g:X\to X$
satisfying $\sup_{x\in X}d(f(x),g(x))<\delta$ there is a continuous map
$h: X\to X$ such that $\sup_{x\in X}d(x,h(x))<\epsilon$ and $h\circ g=f\circ h$.
Every topologically stable homeomorphism of a compact manifold has the POTP,
this is the shadowing lemma by proved Walters \cite{w} (in dimension $2$ or higher) and Morimoto \cite{m} (in dimension $1$).
From this and Theorem \ref{thA} we obtain the following corollary.

\begin{cor}
Every topologically stable homeomorphism with positively expansive measures
of a compact manifold has positive topological entropy.
\end{cor}

Theorem \ref{thA} will be obtained
from the fact that every homeomorphism with positively expansive measures
of a compact metric space also has positively expansive {\em invariant} measures.
This result was mentioned with sketch of the proof in \cite{ms}.
We combine this ingredient with some recent properties
of continuous map with POTP on compact spaces \cite{M}.

This theorem suggests the following question:
Does every continuous map with the POTP$_+$ exhibiting
positively expansive measures of a compact metric space have positive topological entropy?
We can also ask if this theorem is true replacing the POTP by another type of tracing property (see \cite{p} for more of these properties).
In light of \cite{cm} it is natural to pursue our results from discrete to continuous dynamical systems.

\section{Proof of Theorem \ref{thA}}

\noindent
We start with the following simple observation about maps $f: X\to X$ on metric spaces $X$:
$$
f(\Phi_\delta(x))\subset\Phi_\delta(f(x)),
\quad\quad\forall (x,\delta)\in X\times \mathbb{R}^+.
$$
On the other hand, for any Borel measure $\mu$ of $X$ we define
$f_*(\mu)=\mu\circ f^{-1}$.
From the above inclusion we obtain easily the following lemma.

\begin{lemma}
\label{exists-1}
Let $f: X\to X$ be a homeomorphism of a metric space $X$.
If $\mu$ is a positively expansive measure with expansivity constant $\delta$ of $f$,
then so does $f^{-1}_*\mu$.
\end{lemma}

Another useful observation is as follows.
Given a map $f: X\to X$, $x\in X$, $\delta>0$ and $n\in \mathbb{N}^+$ we define
$$
V[x,n,\delta]=\{y\in X:d(f^i(x),f^i(y))\leq \delta,\mbox{ for all } 0\leq i\leq n\},
$$
i.e.,
$$
V[x,n,\delta]=\displaystyle\bigcap_{i=0}^nf^{-i}(B[f^i(x),\delta]),
$$
where $B[\cdot,\cdot]$ denotes the closed ball operation.
It is clear that
$$
\Phi_\delta(x)=\bigcap_{n\in \mathbb{N}^+}V[x,n,\delta]
$$
and that $V[x,n,\delta]\supset V[x,m,\delta]$ for $n\leq m$.
Consequently,
\begin{equation}
\label{epa0}
\mu(\Phi_\delta(x))=\lim_{l\to\infty}\mu(V[x,k_l,\delta])
\end{equation}
for every $x\in X$, $\delta>0$, every Borel probability measure $\mu$ of $X$, and every sequence $k_l\to \infty$.
From this we have the following lemma.

\begin{lemma}
 \label{suff-hom}
Let $f: X\to X$ be a homeomorphism of a metric space $X$.
A Borel probability measure $\mu$ is a positively expansive measure
if and only if there is
$\delta>0$ such that
$$
\liminf_{n\to\infty}\mu(V[x,n,\delta])=0,
\quad\quad\mbox{ for all } x\in X.
$$
\end{lemma}

We shall use this information in the following lemma.
Recall that a Borel measure is {\em invariant} for a map $f$ if $\mu=\mu\circ f^{-1}$.

\begin{lemma}
\label{exists-2}
If $f: X\to X$ is a homeomorphism of a metric space $X$,
then every invariant measure of $f$ which is the limit (with respect to the weak-* topology\index{Topology!weak-*})
of a sequence of positively
expansive measures with a common expansivity constant is positively expansive.
\end{lemma}

\begin{proof}
Denote by $Cl(\cdot)$ and $Int(\cdot)$ the closure and interior operation in $X$.
Denote also by $\partial A=Cl(A)\setminus Int(A)$ the boundary of a subset $A\subset X$.
Let $\mu$ be an invariant probability measure of $f$.
As in the proof of Lemma 8.5 p. 187 in \cite{hk} for all $x\in X$ we can find $\frac{\delta}{2}<\delta_x<\delta$ such that
$$
\mu(\partial(B[x,\delta_x]))=0.
$$
This allows us to define
$$
W[x,n]=\displaystyle\bigcap_{i=0}^nf^{-i}(B[f^i(x),\delta_{f^i(x)}]),
\quad\quad\forall (x,n)\in X\times\mathbb{N}.
$$
Since $\frac{\delta}{2}<\delta_x<\delta$, we can easily verify that
\begin{equation}
\label{epa1}
V\left[x,n,\frac{\delta}{2}\right]\subset W[x,n]\subset V[x,n,\delta], \quad\quad\forall (x,n)\in X\times \mathbb{N}.
\end{equation}
Moreover,
since $f$ (and so $f^{-i}$) are homeomorphisms, one has
$$
\partial(W[x,n])=\partial\left(\displaystyle\bigcap_{i=0}^nf^{-i}(B[f^i(x),\delta_{f^i(x)}])\right)
\subset
$$
$$
\displaystyle\bigcup_{i=0}^n\partial\left(
f^{-i}(B[f^i(x),\delta_{f^i(x)}])\right)
=
\displaystyle\bigcup_{i=0}^nf^{-i}\left(
\partial(B[f^i(x),\delta_{f^i(x)}])\right),
$$
and, since $\mu$ is invariant,
$$
\mu(\partial(W[x,n]))
\leq
\displaystyle\sum_{i=0}^n
\mu(f^{-i}\left(
\partial(B[f^i(x),\delta_{f^i(x)}])\right))
=
$$
$$
\displaystyle\sum_{i=0}^n
\mu(\partial(B[f^i(x),\delta_{f^i(x)}]))=0.
$$
We conclude that
\begin{equation}\label{epa2}
\mu(\partial(W[x,n]))=0,\quad\quad \forall (x,n)\in X\times \mathbb{N}.
\end{equation}

Now, suppose that $\mu$ is the weak-* limit of a sequence of positively expansive measures $\mu_n$ with a common expansivity
constant $\delta$.
Clearly, $\mu$ is also a probability measure.
Fix $x\in X$. 
Since each $\mu_n$ is a probability we have
$0\leq \mu_m(W[x,n])\leq 1$
for all $n,m\in \mathbb{N}$.
Then, we can apply
the Bolzano-Weierstrass Theorem\index{Theorem!Bolzano-Weierstrass} to find sequences $k_l, r_s\to\infty$ for which the double limit
$$
\displaystyle\lim_{l,s\to\infty}\mu_{r_s}(W[x,k_l])
$$
exists.

On the one hand, for fixed $l$, using (\ref{epa2}), $\mu_n\to \mu$ and well-known properties of the weak-* topology
(e.g. Theorem 6.1-(e) p. 40 in \cite{pa}) one has that the limit
$$
\lim_{s\to\infty}\mu_{r_s}(W[x,k_l])=\mu(W[x,k_l])
$$
exists.

On the other hand, the second inequality in (\ref{epa1}) and (\ref{epa0}) imply for fixed $s$ that
$$
\lim_{l\to\infty}\mu_{r_s}(W[x,k_l])\leq \lim_{l\to\infty}\mu_{r_s}(V[x,k_l,\delta])
=\mu_{r_s}(\Gamma_\delta(x))=0.
$$
Consequently, the limit
$$
\lim_{l\to\infty}\mu_{r_s}(W[x,k_l])=0
$$
also exists for fixed $s$.

From these assertions and well-known properties of double sequences one obtains
$$
\lim_{l\to\infty}\lim_{s\to\infty}\mu_{r_s}(W[x,k_l])=\lim_{s\to\infty}\lim_{l\to\infty}\mu_{r_s}(W[x,k_l])=0.
$$
But (\ref{epa1}) implies
$$
\liminf_{n\to\infty}\mu\left(V\left[x,n,\frac{\delta}{2}\right]\right)\leq\lim_{l\to\infty}\mu\left(V\left[x,k_l,\frac{\delta}{2}\right]\right)\leq
\lim_{l\to\infty}\mu(W[x,k_l])
$$
and $\mu_n\to \mu$ together with (\ref{epa2}) yields
$$
\lim_{l\to\infty}\mu(W[x,k_l])
=
\lim_{l\to\infty}\lim_{s\to\infty}\mu_{r_s}(W[x,k_l])
$$
so
$$
\liminf_{n\to\infty}\mu\left(V\left[x,n,\frac{\delta}{2}\right]\right)=0
$$
and then $\mu$ is positively expansive by Lemma \ref{suff-hom}.
\end{proof}

These results imply the following lemma.

\begin{lemma}
\label{exists-3}
A homeomorphisms of a compact metric space has positively
expansive measures if and only if it has positively expansive {\em invariant} measures.
\end{lemma}

\begin{proof}
Let $\mu$ be a positively expansive measure with expansivity constant $\delta$ of a homeomorphism $f: X\to X$ of a compact metric space $X$.
By Lemma \ref{exists-1} we have that $f^{-1}_*\mu$ is a positively expansive measure with positive expansivity constant $\delta$ of $f$.
Therefore,
$f^{-i}_*\mu$ is a positively expansive measure with positively expansivity constant $\delta$ of $f$ ($\forall i\in \mathbb{N}$), and so,
$$
\mu_n=\frac{1}{n}\displaystyle\sum_{i=0}^{n-1}f^{-i}_*\mu,
\quad\quad n\in\mathbb{N}^+
$$
is a sequence of positively expansive measures with common expansivity constant $\delta$.
As $X$ is compact there is a subsequence $n_k\to\infty$ such that
$\mu_{n_k}$ converges to a Borel probability measure $\mu$. Since $\mu$ is clearly invariant for $f^{-1}$ and $f$ is a homeomorphism,
we have that $\mu$ is also an invariant measure of $f$.
Then, we can apply
Lemma \ref{exists-2} to this sequence to obtain that $\mu$ is
a positively expansive measure of $f$.
\end{proof}

Recall that a map $f$ of a metric space is {\em equicontinuouus}
if for every $\epsilon>0$ there is $\delta>0$ such that
if $x,y\in X$ satisfy $d(x,y)<\delta$, then $d(f^i(x),f^i(y))<\epsilon$ for every $i\in\mathbb{N}$.
The following lemma is a particular case of Corollary 6 in \cite{M}.

\begin{lemma}
 \label{l3}
If $f$ is a homeomorphism with the POTP$_+$ of a compact metric space and $h(f)=0$, then $f|_{\Omega(f)}$ is equicontinuous.
\end{lemma}

It is well-known that a equicontinuous homeomorphism of a compact metric space has zero topological entropy. The following lemma improves
this result as follows.

\begin{lemma}
\label{equicontinuous}
An equicontinuous homeomorphism of a compact metric space has no
positively expansive measures.
\end{lemma}

\begin{proof}
Let $f$ be a equicontinuous homeomorphism
of a compact metric space.
Suppose by contradiction that $f$ has a positively expansive measure
$\mu$.
Take an expansivity constant $e$ of $\mu$. 
Putting it in the definition of equicontinuity
we obtain $\delta>0$ such that $B[x,\delta]\subset \Phi_e(x)$ for every $x\in X$.
From this we obtain $\mu(B[x,\delta])=0$ for every $x\in\Omega(f)$.
But $X$ is compact so there are finitely many points $x_1,\cdots, x_k$
such that
$X=\bigcup_{i=1}^kB[x_i,\delta]$.
Then,
$$
\mu(X)\leq
\displaystyle\sum_{i=1}^k\mu(B[x_i,\delta])=0
$$
which is absurd.
This concludes the proof.
\end{proof}

\begin{proof}[Proof of Theorem \ref{thA}]
Let $f$ be a homeomorphism of a compact metric space
such that $f|_{\Omega(f)}$ has the POTP.
Suppose by contradiction that $f$ has a positively expansive measure but $h(f)=0$.
By Lemma \ref{exists-3} we have that 
$f$ has a positively expansive invariant measure $\mu$.
Since $\mu$ is invariant, we have that $\mu$ is supported on
$\Omega(f)$.
On the other hand, POTP and POTP$_+$ are equivalent among homeomorphisms \cite{p}.
Since $h(f)=0$, we conclude that $f|_{\Omega(f)}$ is equicontinuous by Lemma \ref{l3}.
This and the existence of $\mu$ contradicts Lemma \ref{equicontinuous}.
\end{proof}

\end{document}